\documentclass[letterpaper,10pt,reqno]{amsart}
\usepackage{indentfirst} 
\usepackage{amssymb}
\usepackage{mathtools} 
\usepackage{mathabx} 
\usepackage{amsthm}
\usepackage{thmtools}
\usepackage{enumitem} 
\usepackage[colorlinks=true]{hyperref} 
\usepackage[usenames,dvipsnames]{xcolor} 
\usepackage{ifthen}
\usepackage{tikz}
\usetikzlibrary{decorations.pathreplacing}
\usepackage{tikz-cd}
  
\makeatletter

\def\MRbibitem{\@ifnextchar[\my@lbibitem\my@bibitem}

\def\mybiblabel#1#2{\@biblabel{{\hyperref{http://www.ams.org/mathscinet-getitem?mr=#1}{}{}{#2}}}}

\def\myhyperanchor#1{\Hy@raisedlink{\hyper@anchorstart{cite.#1}\hyper@anchorend}}

\def\my@lbibitem[#1]#2#3#4\par{%
  \item[\mybiblabel{#2}{#1}\myhyperanchor{#3}\hfill]#4%
  \@ifundefined{ifbackrefparscan}{}{\BR@backref{#3}}%
  \if@filesw{\let\protect\noexpand\immediate
    \write\@auxout{\string\bibcite{#3}{#1}}}\fi\ignorespaces%
}

\def\my@bibitem#1#2#3\par{%
  \refstepcounter\@listctr
  \item[\mybiblabel{#1}{\the\value\@listctr}\myhyperanchor{#2}\hfill]#3%
  \@ifundefined{ifbackrefparscan}{}{\BR@backref{#2}}%
  \if@filesw\immediate\write\@auxout
    {\string\bibcite{#2}{\the\value\@listctr}}\fi\ignorespaces%
}

\makeatother

\DeclareFontFamily{U} {MnSymbolA}{}
\DeclareFontShape{U}{MnSymbolA}{m}{n}{
   <-6> MnSymbolA5
   <6-7> MnSymbolA6
   <7-8> MnSymbolA7
   <8-9> MnSymbolA8
   <9-10> MnSymbolA9
   <10-12> MnSymbolA10
   <12-> MnSymbolA12}{}
\DeclareFontShape{U}{MnSymbolA}{b}{n}{
   <-6> MnSymbolA-Bold5
   <6-7> MnSymbolA-Bold6
   <7-8> MnSymbolA-Bold7
   <8-9> MnSymbolA-Bold8
   <9-10> MnSymbolA-Bold9
   <10-12> MnSymbolA-Bold10
   <12-> MnSymbolA-Bold12}{}
\DeclareSymbolFont{MnSyA} {U} {MnSymbolA}{m}{n}
 \DeclareFontFamily{U} {MnSymbolC}{}
\DeclareFontShape{U}{MnSymbolC}{m}{n}{
  <-6> MnSymbolC5
  <6-7> MnSymbolC6
  <7-8> MnSymbolC7
  <8-9> MnSymbolC8
  <9-10> MnSymbolC9
  <10-12> MnSymbolC10
  <12-> MnSymbolC12}{}
\DeclareFontShape{U}{MnSymbolC}{b}{n}{
  <-6> MnSymbolC-Bold5
  <6-7> MnSymbolC-Bold6
  <7-8> MnSymbolC-Bold7
  <8-9> MnSymbolC-Bold8
  <9-10> MnSymbolC-Bold9
  <10-12> MnSymbolC-Bold10
  <12-> MnSymbolC-Bold12}{}
\DeclareSymbolFont{MnSyC} {U} {MnSymbolC}{m}{n}

\DeclareMathSymbol{\top}{\mathord}{MnSyA}{219} 
\DeclareMathSymbol{\plus}{\mathord}{MnSyC}{20} 

\declaretheorem[numberwithin=section]{theorem}
\declaretheorem[sibling=theorem]{lemma}
\declaretheorem[sibling=theorem]{corollary}
\declaretheorem[sibling=theorem]{proposition}
\declaretheorem[sibling=theorem,style=definition]{definition}

\declaretheorem[sibling=theorem,style=remark]{remark}

\hypersetup{bookmarksdepth = 3} 
\numberwithin{equation}{section}     

\setlist[enumerate,1]{label={\upshape(\alph*)},ref=\alph*}
\setlist[enumerate,2]{label={\upshape(\arabic*)},ref=\arabic*}

\newcommand{\R}{\mathbb{R}}
\newcommand{\Z}{\mathbb{Z}}
\newcommand{\N}{\mathbb{N}}

\newcommand{\M}{\mathcal{M}}

\newcommand{\Rnon}{\mathbb{R}_{\plus}}    
\newcommand{\Rpos}{\mathbb{R}_{\plus\plus}} 
\newcommand{\Mat}[2][]{\ifthenelse{\equal{#1}{}}{\R^{{#2}\times{#2}}}{\R^{{#1}\times{#2}}}}
\newcommand{\Man}[2][]{\ifthenelse{\equal{#1}{}}{\Rnon^{{#2}\times{#2}}}{\Rnon^{{#1}\times{#2}}}}
\newcommand{\Map}[2][]{\ifthenelse{\equal{#1}{}}{\Rpos^{{#2}\times{#2}}}{\Rpos^{{#1}\times{#2}}}}

\renewcommand{\epsilon}{\varepsilon}
\renewcommand{\phi}{\varphi}
\renewcommand{\setminus}{\smallsetminus}

\begin{document}

\title{Upper semi-continuity of entropy in non-compact settings}
\date{\today}


\begin{thanks}
{We would like to thank Jerome Buzzi, Sylvain Crovisier and Omri Sarig for useful comments and remarks. 
G.I.\ was partially supported by CONICYT PIA ACT172001 and by Proyecto Fondecyt 1150058.}
\end{thanks}

\author[G.~Iommi]{Godofredo Iommi} \address{Facultad de Matem\'aticas,
Pontificia Universidad Cat\'olica de Chile (PUC), Avenida Vicu\~na Mackenna 4860, Santiago, Chile}
\email{giommi@mat.puc.cl}
\urladdr{http://www.mat.uc.cl/~giommi/}
\author[M.~Todd]{Mike Todd}
\address{Mathematical Institute,
University of St Andrews,
North Haugh,
St Andrews,
KY16 9SS,
Scotland} 
\email{m.todd@st-andrews.ac.uk}

\urladdr{http://www.mcs.st-and.ac.uk/~miket/}
 \author[A.~Velozo]{An\'ibal Velozo}  \address{Department of Mathematics, Yale University, New Haven, CT 06511, USA.}
\email{anibal.velozo@yale.edu }
\urladdr{https://gauss.math.yale.edu/~av578/}

\maketitle

\begin{abstract}
We prove that the entropy map for countable Markov shifts of finite entropy is upper semi-continuous at ergodic measures. Note that the phase space is non-compact.
  We also discuss the related problem of existence of measures of maximal entropy.
  \end{abstract}

\section{Introduction}

The entropy map of  the continuous transformation $T:X \to X$ defined on a metric space $(X,d)$ is the map $\mu \mapsto h_{\mu}(T)$ which is defined on the space of $T-$invariant probability measures $\M_T(X)$, where $h_{\mu}(T)$ is the entropy of $\mu$ (precise definitions can be found in Section \ref{sec:pre}).  The study of the continuity properties of the entropy map, with $\M_T(X)$ endowed with the weak$^*$  topology,  goes back at least to the work of Bowen \cite{bo1}.  In general it is not a continuous map (see \cite[p.184]{wa}). However, in certain relevant cases it can be shown that the map is upper semi-continuous. For example, if $T$ is an expansive homemorphism of a compact metric space then the entropy map is upper semi-continuous, see \cite[Theorem 8.2]{wa} and \cite{bo1,de,mi}. It is also known that if $T$ is a $C^{\infty}$ map defined over a smooth compact manifold  then, again, the entropy map is upper semi-continuous (see \cite[Theorem 4.1]{n} and \cite{y}). Lyubich \cite[Corollary 1]{ly}  proved that the entropy map is upper semi-continuous for rational maps of the Riemann sphere. In all the above examples the phase space is compact: in this article we will drop the compactness assumption on the underlying space.

Markov shifts defined over finite alphabets have been used with remarkable success to study uniformly hyperbolic systems. Indeed, these systems possess finite Markov partitions and are, therefore, semi-conjugated to Markov shifts. See \cite{bo2} for an example of the wealth of results that can be obtained with this method. Following the 2013 work of Sarig \cite{sa3}, countable Markov partitions have been constructed for a wide range of dynamical systems defined on compact spaces  (see \cite{bu, lm, ls, o}). The symbolic coding captures a relevant part, though not all, of the dynamics. For example, diffeomorphisms defined on compact manifolds  have countable Markov partitions that capture all hyperbolic measures. These  results have prompted a renewed interest in the ergodic properties of countable Markov shifts.  

In this paper
we study countable Markov shifts $(\Sigma, \sigma)$. More precisely, the shift on sequences of elements of a countable alphabet in which the transitions, allowed and forbidden, are described by a directed graph (see Section \ref{sec:cms} for definitions). The space $\Sigma$ is not compact with respect to its natural topology. If the alphabet is finite, and therefore the space $\Sigma$ is compact,  it is a classical result that the entropy map is upper semi-continuous \cite[Theorem 8.2]{wa}.   In this article we recover upper semi-continuity in the ergodic case for shifts of finite topological entropy defined on non-compact spaces.
Our main result is:

\begin{theorem} \label{thm:usc}
Let $(\Sigma, \sigma)$ be a  two-sided countable Markov shift of finite topological entropy. Then the entropy map is upper  semi-continuous on the set of ergodic measures.  
\end{theorem}

This has the following corollary: 

\begin{corollary} \label{cor:uscN}
If $(\Sigma^{+},\sigma)$ is a one-sided  countable Markov shift of finite topological entropy. Then the  entropy map is upper semi-continuous on the set of ergodic measures.
\end{corollary}

The strategy of the proof is the following. First note that if there exists a finite generating partition of the space such that the measure of its boundary is zero for every invariant measure, then the entropy map is upper semi-continuous. This is no longer true for countable generating partitions. Krieger \cite{kr}  constructed a finite generating partition for each ergodic measure. More recently, Hochman \cite{h1,h2} constructed finite partitions that are generating for every ergodic  measure. We can not use the result by Krieger  since we need the same partition for every measure and we can not use Hochman's result since his partitions have a large boundary. We overcome this difficulty by constructing finite generating partitions with no boundary.

If  $T:X \to X$ is a continuous transformation defined on a compact space for which the entropy map is upper semi-continuous then there exists a measure of maximal entropy. Diffeomorphisms of class $C^r$, for any $r \in [1, \infty)$ (in the compact setting), with no measure of maximal entropy have been constructed by  Misiurewicz \cite{mi1} and Buzzi \cite{bu1}. These are examples for which the entropy map is not upper semi-continuous. If the space $X$ is non-compact, even if the entropy map is upper semi-continuous, measures of maximal entropy might not exist. Indeed, in the non-compact case the space of invariant probability measures might also be non-compact and therefore a sequence of measures with entropy converging to the topological entropy may not have a convergent subsequence. For instance, consider the geodesic flow on a non-compact  pinched negatively curved manifold.  Velozo \cite{v} showed that in that setting the entropy map is always upper semi-continuous, but there are examples for which there is no measure of maximal entropy (see for example \cite{dpps}). In Section \ref{sec:mme} we discuss conditions that guarantee the existence of measures of maximal entropy. We also address the relation between upper semi-continuity of the entropy map and thermodynamic formalism. More precisely,  we provide conditions for the existence of equilibrium measures.   Finally, in Section  \ref{sec:flo} we apply our results to suspension flows.

We would expect our results to pass to some systems which are coded by countable Markov shifts, but with the important caveat that it is not always straightforward to compare topologies on the shift versus the coded space, so weak$^*$ convergence in one space may not imply weak$^*$ convergence in the other.

\section{Entropy and generators} \label{sec:pre}
This  section is devoted to recall basic properties and definitions that will be used throughout the article. The reader is referred to \cite{d, p,wa} for more details. 

\subsection{The weak$^*$  topology}
Let $(X,d)$ be a metric space, we denote by $C_b(X)$ the space bounded continuous functions $\phi:X\to \R$. Denote by $\M(X)$ the set of Borel probability measures on the metric space $(X,d)$. 

\begin{definition} \label{def:wc}
A sequence of probability measures $(\mu_n)_n$ defined on a metric space $(X,d)$ converges to a measure $\mu$ in the weak$^*$  topology if for every $\phi \in C_b(X)$ we have 
\begin{equation*}
\lim_{n \to \infty} \int \phi~d \mu_n = \int \phi~d \mu.
\end{equation*}
\end{definition}

\begin{remark} \label{rem:lip}
In this notion of convergence we can replace the set of test functions  by  the space of bounded Lipschitz functions (see \cite[Theorem 13.16 (ii)]{kl}). That is, if for every bounded Lipschitz function $\phi:X\to \R$ we have 
\begin{equation*}
\lim_{n \to \infty} \int \phi~d \mu_n = \int \phi~d \mu.
\end{equation*}
then the sequence $(\mu_n)_n$ converges in the weak$^*$  topology to $\mu$.
\end{remark}

If the space $(X,d)$ is compact then  so is $\M(X)$ with respect to the weak$^*$  topology (see \cite[Theorem 6.4]{par}). For $A \subset X$ denote by  $\text{int }A ,  \overline{A}$ and $\partial A$ the interior, the closure and   the boundary of the set $A$, respectively. The following result characterises the weak$^*$  convergence (see \cite[Theorem 6.1]{par}), 

\begin{theorem}[Portmanteau Theorem] \label{thm:port}
Let $(X,d)$ be a metric space and $(\mu_n)_n, \mu $ measures  in $\M(X)$. The following statements are equivalent:
\begin{enumerate}
\item The sequence $(\mu_n)_n$ converges in the  weak$^*$  topology to the measure $\mu$.
\item If $C \subset X$ is a closed set then $\limsup_{n \to \infty} \mu_n(C) \leq \mu(C)$.
\item If $O \subset X$ is an open set then $\liminf_{n \to \infty} \mu_n(C) \geq \mu(C)$.
\item If $A \subset X$ is a set  such that $\mu (\partial A)=0$  then $\lim_{n \to \infty} \mu_n(A) = \mu(A)$.
\end{enumerate}
\end{theorem}

Let $T:(X,d) \to (X,d)$ be a continuous dynamical system,  denote by $\M_T$ the space of $T-$invariant probability measures. 

\begin{proposition} \label{prop_inv}
Let $T:(X,d) \to (X,d)$ be a continuous dynamical system defined on a metric space, then
\begin{enumerate}
\item The space $\M_T$ is closed in the weak$^*$  topology (\cite[Theorem 6.10]{wa}).
\item  If $X$ is compact then so is $\M_T$  with respect to the weak$^*$  topology (see \cite[Theorem 6.10]{wa}).
\item The space $\M_T$ is a convex set for which its extreme points are the ergodic measures  (see \cite[Theorem 6.10]{wa}). It is actually a Choquet simplex (each measure is represented in a unique way as a generalized convex combination of the ergodic measures  \cite[p.153]{wa}).
\end{enumerate}
\end{proposition}

\subsection{Entropy of a dynamical system}

Let $T:(X,d) \to (X,d)$ be a continuous dynamical system. We recall the definition of entropy of an invariant measure $\mu \in \M_T$ (see \cite[Chapter 4]{wa} for more details).

\begin{definition}
A partition $\mathcal{P}$ of a probability space  $(X, \mathcal{B}, \mu)$ is a countable (finite or infinite) collection of pairwise disjoint subset of $X$ whose union has full measure.
\end{definition}

\begin{definition} The \emph{entropy} of the partition $\mathcal{P}$ is defined by
\begin{equation*}
H_\mu(\mathcal{P}):= - \sum_{P \in \mathcal{P}} \mu(P) \log \mu(P),
\end{equation*}
where $0 \log 0 :=0$. 
\end{definition}
It is possible that $H_\mu(\mathcal{P})=\infty$. Given two partitions $\mathcal{P}$ and $\mathcal{Q}$ of $X$ we define the new partition 
\begin{equation*}
\mathcal{P} \vee \mathcal{Q}:= \left\{P \cap Q : P \in  \mathcal{P} , Q \in  \mathcal{Q} \right\}
\end{equation*}
Let  $\mathcal{P}$ be a partition of $X$ we define the partition $T^{-1}\mathcal{P}:= \left\{ T^{-1}P : P \in \mathcal{P} \right\}$ and for $n \in \N$ we set  
$\mathcal{P}^n:=\bigvee_{i=0}^{n-1} T^{-i}\mathcal{P}$.  Since the measure $\mu$ is $T-$invariant, the sequence $H_{\mu}(\mathcal{P}^n)$ is sub-additive. 
\begin{definition}
The \emph{entropy} of $\mu$ with respect to $\mathcal{P}$ is defined by
\begin{equation*}
h_{\mu}(\mathcal{P}):=  \lim_{n \to\infty} \frac{1}{n} H_{\mu}(\mathcal{P}^n)
\end{equation*}
\end{definition}

\begin{definition}
The \emph{entropy} of $\mu$ is defined by
\begin{equation*}
h_{\mu}(T):= \sup \left\{h_{\mu}(\mathcal{P}) : \mathcal{P} \text{ a partition with } H_{\mu}(\mathcal{P}) < \infty		\right\}.		
\end{equation*}
\end{definition}
If the underlying dynamical system considered is clear we write $h_{\mu}$ instead of $h_{\mu}(T)$.

\subsection{Generators}
We now recall the definition and properties of an important concept in ergodic theory, namely generators.

\begin{definition}
Let $(T, X, \mathcal{B}, \mu)$ be a dynamical system. A \emph{one-sided generating partition} $\mathcal{P}$ of  $(T,X,\mathcal{B}, \mu)$ is a partition such that $\bigcup_{n=1}^{\infty} \mathcal{P}^n$ generates the sigma-algebra  $\mathcal{B}$ up to sets of measure zero. Analogously, a  \emph{two-sided generating partition} $\mathcal{P}$ of  $(T,X,\mathcal{B}, \mu)$ is a partition such that $\bigcup_{n=-\infty}^{\infty} \mathcal{P}^n$ generates the sigma-algebra  $\mathcal{B}$ up to sets of measure zero. \end{definition}

 A classical result by Kolmogorov and Sinai \cite{si} states that entropy can be computed with either one- or two-sided generating partitions (see also \cite[Theorem 4.17 and 4.18]{wa} or \cite[Theorem 4.2.2]{d}).
 
 \begin{theorem}[Kolmogorov-Sinai] 
 Let $(T, X, \mathcal{B}, \mu)$ be a dynamical system  and  $\mathcal{P}$ a one-or a two-sided generating partition, then $h_{\mu}(T)= h(\mu, \mathcal{P})$.
   \end{theorem}

The existence of generating partitions depends on the acting semi-group. Rohlin \cite{ro1, ro2} proved that ergodic (actually aperiodic) invertible systems of finite entropy have countable two-sided generators. This was later improved by Krieger \cite{kr} (see also \cite[Theorem 4.2.3]{d}).

\begin{theorem}[Krieger]
Let $(T, X, \mathcal{B}, \mu)$ be an ergodic  invertible dynamical system with $h_\mu(T) <\infty$ then there exists a finite two-sided generator $\mathcal{P}$ for $\mu$. The cardinality of $\mathcal{P}$ can be chosen to be any integer larger than $e^{h_{\mu}(T)}$.
\end{theorem}

\begin{remark}
The above result is false for dynamical systems in which the action is given by $\N$. Non-invertible dynamical systems do not, in general, have finite generators. Existence of finite generating partitions depends on the acting semi-group, as seen the result is valid for $\Z$ and false for $\N$. It has been shown that finite generating partitions do exist also when the acting group is amenable \cite{dp, ro}. 
\end{remark}

Note that in Krieger's result the generator depends upon the measure. There are some well known cases in which there exists a finite partition which is a generator  for \emph{every} measure. For example, if $(\Sigma, \sigma)$ is a transitive two-sided sub-shift of finite type on a finite alphabet then the cylinders of length one form a two-sided generating partition for every invariant measure. Hochman  
proved that a uniform version of Krieger result can be obtained (see \cite[Corollary 1.2]{h2}).

\begin{theorem}[Hochman] \label{thm:hoch}
Let $(T, X, \mathcal{B})$ be an invertible  dynamical system with no periodic points and of finite entropy. Then there exists a finite two-sided partition $\mathcal{P}$ that is a generator for every ergodic  invariant measure.
\end{theorem}

\subsection{Upper semi-continuity of the entropy} 
 The following well known result (see for example \cite[Lemma 6.6.7]{d}) relates the existence of finite generating partitions with continuity properties of the entropy map.

\begin{proposition} \label{prop:usc}
Let $(T, X, \mathcal{B})$ be a dynamical system and $\mathcal{P}$ be a finite two- sided generator for every measure in $\mathcal{C}\subset \mathcal{M}_T$. If for every  $\mu \in \mathcal{C}$  we have that $\mu(\partial \mathcal{P})=0$ then the entropy map is upper semi-continuous in $\mathcal{C}$.
\end{proposition}	
	
\begin{proof}
Let $\mu \in \mathcal{C}$  and $P \in \mathcal{P}$. Note that since  $\mu(\partial \mathcal{P})=0$ the function $\nu \to \nu(P)$ is continuous at $\mu$.  Since the partition $\mathcal{P}$ is finite, the function defined in $\mathcal{C}$ by
\begin{equation*}
\nu \to H_{\nu}(\mathcal{P})= - \sum_{P \in \mathcal{P}} \nu(P) \log \nu(P),
\end{equation*}
is continuous at $\mu$. Note that  $\mu(\partial \mathcal{P})=0$ implies that $\mu(\partial \mathcal{P}^n)=0$. Thus, the function in $\mathcal{C}$ defined by
$\nu \to H_{\nu}(\mathcal{P}^n)$ is also continuous. Since the function defined in $\mathcal{C}$ by
\begin{equation*}
\nu \mapsto h_{\nu}(\mathcal{P})= \inf_n \frac{1}{n} H_{\nu}(\mathcal{P}^n)
\end{equation*}
is the infimum of continuous functions at $\mu$ we have that the map $\nu \mapsto h_{\nu}(\mathcal{P})$ is upper semi-continuous at $\mu$. Since $\mathcal{P}$ is a uniform generating partition in $\mathcal{C}$ for every  $\nu \in \mathcal{C}$ we have $h_{\nu}(T)=h_{\nu}(\mathcal{P})$. Thus, the map
\begin{equation*}
\nu \mapsto h_{\nu}(T)
\end{equation*}
is upper semi-continuous at $\mu$. Since $\mu \in \mathcal{C}$ was arbitrary the result follows. 
\end{proof}

\begin{remark}
Note that if in Proposition \ref{prop:usc} we have $\mathcal{C}=\mathcal{M}_T$ then the entropy map is upper semi-continuous in the space of invariant probability measures.
\end{remark}

\begin{remark}
The argument in Proposition \ref{prop:usc} breaks down if we consider countable (infinite) generating partitions since, in that case, the map $\nu \to H_{\nu}(\mathcal{P})$ need not to be continuous (see Remark \ref{rem:infentropy} or \cite[p.774]{jmu}). 
\end{remark}

\begin{remark} Expanding maps defined on compact metric spaces are examples of dynamical systems having finite generators as in Proposition \ref{prop:usc} (see \cite[Theorem 8.2]{wa}). We stress that while  Theorem~\ref{thm:hoch} provides a finite (uniform) generating partition, in general this does not satisfy the condition of zero measure boundary.
\end{remark}

\section{Countable Markov shifts}  \label{sec:cms}
In this section we define countable Markov shifts, both one-and two-sided and prove our main results.

\subsection{Two-sided countable Markov shifts}
Let $(\Sigma, \sigma)$ be a transitive two-sided Markov shift defined over a countable alphabet $\mathcal{A}$. This means that there exists a matrix $S=(s_{ij})_{\mathcal{A} \times \mathcal{A} }$ of zeros and ones such that
\begin{equation*}
\Sigma=\left\{ x \in \mathcal{A} ^{\mathbb{Z}} : s_{x_{i} x_{i+1}}=1 \ \text{for every $i \in \mathbb{Z}$}\right\}.
\end{equation*}
Note that the matrix $S$ induces a directed graph on $\mathcal{A}$. Let $n \in \N$ and ${\bf{r}}:=(r_1, \dots, r_n)\in \mathcal{A}^n$, we say that ${\bf{r}}$ is an \emph{admissible word} if $s_{r_{i} r_{i+1}}=1,$ for $i \in \{1, \dots, n-1\}$. In this setting transitivity means that given $x,y\in \mathcal{A}$, there exists an admissible word starting at $x$ and ending at $y$.  Let $(r_1, \dots, r_n)$ be an admissible word and $l\in \Z$, we define the corresponding cylinder set  by
\begin{equation*}
[r_1,...,r_n]_l := \left\{x \in \Sigma: x_l=r_1, x_{l+1}=r_2,...,x_{l+n-1}=r_n  \right\}. 
\end{equation*}
We endow $\Sigma$ with the topology generated by the cylinder sets. Note that, with respect to this topology, the space $\Sigma$ is  non-compact.   The \emph{shift map} $\sigma:\Sigma \to \Sigma$ is defined by $(\sigma(x))_i=x_{i+1}$. Let $\mathcal{M}_{\sigma}$ be the set of $\sigma-$invariant probability measures and  $\mathcal{E}_{\sigma} \subset \mathcal{M}_{\sigma}$ the set of ergodic invariant probability measures. The topological entropy of $\sigma$ is defined by
\begin{equation*}
h_{top}(\sigma):= \lim_{n \to \infty} \frac{1}{n} \log \sum_{\sigma^nx =x} \chi_{[a]_0}(x) = \sup \left\{h_{\mu}(\sigma) : \mu \in \mathcal{M}_{\sigma}		\right\}= \sup \left\{h_{\mu}(\sigma) : \mu \in \mathcal{E}_{\sigma}		\right\},
\end{equation*}
where $a \in \mathcal{A}$ is an arbitrary symbol and $\chi_{[a]_0}$ is the characteristic function of the cylinder $[a]_0$. This definition was introduced by Gurevich \cite{gu1, gu2} who  proved that the limit exists (see also \cite[Remark 3.2]{ds}),  and
 also proved the second  and third equalities. Since the system is transitive the definition does not depend on the symbol $a$. If the system $(\Sigma, \sigma)$ is not transitive and $h_{var}:=\sup \left\{h_{\mu}(\sigma) : \mu \in \mathcal{E}_{\sigma}		\right\} < \infty$ then the topological entropy of any transitive component is bounded above by $h_{var}$ and we define $h_{top}$ as $h_{var}$.

\subsection{Proof of Theorem \ref{thm:usc} when $(\Sigma,\sigma)$ is transitive} We now assume that $(\Sigma,\sigma)$ is a transitive countable Markov shift. Given $a\in \mathcal{A}$ we define
\begin{equation*}
\Sigma_a:=\{x\in\Sigma: \sigma^k x\in [a] \text{ for infinitely many positive and negative } k \in \mathbb{Z}\}.
\end{equation*}
Observe that $\Sigma_a$ is a Borel $\sigma$-invariant subset of $\Sigma$. Therefore, the dynamical system $\sigma : \Sigma_a \to \Sigma_a$ is well defined. Since $(\Sigma, \sigma)$ is a finite entropy system, so is $(\Sigma_a, \sigma)$. 

\begin{remark} 
Let $\mu \in \mathcal{E}_{\sigma}$ be an ergodic $\sigma$-invariant probability measure  such that  $\mu([a])>0$, then by the Birkhoff ergodic theorem we have $\mu(\Sigma_a)=1$.  This is the only point in our proof where we use ergodicity of our measures of interest. Moreover, since the system is transitive the set $\Sigma_a$ is dense in $\Sigma$.
\end{remark}
The following class of countable Markov shifts that has been studied in \cite{bbg,ru,sa} will be of importance in what follows.

\begin{definition}A \emph{loop graph} is a graph made of simple loops which are based at a common vertex and otherwise do not intersect. A \emph{loop system} is the two-sided countable Markov shift defined by a loop graph.  
\end{definition}

\begin{lemma}\label{topconj} The system  $(\Sigma_a,\sigma)$ is topologically conjugate to a loop system $(\overline\Sigma, \sigma)$ of finite entropy.
\end{lemma}
\begin{proof} For every $n \in \N$ denote by $C_n$ the set of non-empty cylinders in $\Sigma$ of the form $[ax_1...x_na]_0$, where $x_i\ne a$ for all $i \in \{1, \dots, n   \}$. Since the entropy of $\Sigma$ is finite, the number of elements in $C_n$ is finite. Let $c_n$ be the number of elements in $C_n$ and write $C_n=\{A_n^1,...,A_n^{c_n}\}$. Construct a loop graph with exactly $c_n$ loops of length $n+1$, and denote by $\overline{\Sigma}$ the loop system associated to it.
 For convenience we denote the vertex of the loop graph with the letter $a$. 
There is a one  to one correspondence between $C_n$ and the non-empty cylinders in $\overline{\Sigma}$ of the form $[a{\bf x}a]$, where the word ${\bf x}$ does not contain the letter $a$ and it has length $n$. Denote by $B_n^i$ to the cylinder in $\overline{\Sigma}$ associated to $A_n^i$. 

Note that every $x \in \Sigma_a$ is of the following form $(\dots a{\bf x_{-1}}a{\bf x_0}a{\bf x_1}a \dots)$, where  ${\bf x_i}$ are admissible words that do not contain the letter $a$. Observe that for each ${\bf x_i}$ there is a unique corresponding cylinder $A_{n_i}^i$. Moreover, for each $A_{n_i}^i$ there is a unique  corresponding cylinder $B_{n_i}^i$ in $\overline{\Sigma}$. Finally, fo each $B_{n_i}^i$ 
there  is a unique corresponding admissible word $\bf{x}^b_i$ in $\overline{\Sigma}$. Following this procedure we can define a bijective map
$F:\Sigma_a \to \overline{\Sigma}$ by $F(\dots a{\bf x_{-1}}a{\bf x_0}a{\bf x_1}a \dots)=(\dots a{\bf x^b_{-1}}a{\bf x^b_0}a{\bf x^b_1}a \dots).$

We will now prove that $F$ is a homeomorphism. Let $U$ be an open set in $\overline{\Sigma}$. Consider a point $x\in U$ and define $y=F^{-1}(x)$. To prove the continuity of $F$ it is enough to check that $y$ is an interior point of $F^{-1}(U)$.  Let $[a{\bf x^b_{1}}a{\bf x^b_2}a...a{\bf x^b_m}a]_h$, where $h\in\Z$ and no ${\bf x_i}$ contains the letter $a$, be a cylinder contained in $U$ such that $x \in U$. Note that
\begin{equation*}
F^{-1}([a{\bf x^b_{1}}a{\bf x^b_2}a...a{\bf x^b_m}a]_h)=[a{\bf x_{1}}a{\bf x_2}a...a{\bf x_m}a]_h
\end{equation*}
Note that $y \in [a{\bf x_{1}}a{\bf x_2}a...a{\bf x_m}a]_h$,  $[a{\bf x_{1}}a{\bf x_2}a...a{\bf x_m}a]_h \subset F^{-1}U$ and the cylinder  set
$[a{\bf x_{1}}a{\bf x_2}a...a{\bf x_m}a]_h$ is open. Therefore, $F$ is continuous. A similar argument gives that $F^{-1}$ is also continuous, therefore $F$ is a homeomorphism. By construction we have that $\sigma|_{\overline{\Sigma} \circ F} = F \circ \sigma |_{\Sigma_a}$. Since $\Sigma_a$ has finite entropy so does $\overline{\Sigma}$.
\end{proof}

The following result was obtained by Boyle, Buzzi and G\'omez \cite[Lemma 3.7]{bbg}, they established the existence of a \emph{continuous} embedding  of a loop system into a compact sub-shift. Let us stress that the relevant part of the result is the continuity. 
Borel embeddings have been obtained in greater generality (see Hochman \cite[Corollary 1.2]{h2} or \cite[Theorem 1.5]{h1}).

\begin{theorem}[Boyle, Buzzi, G\'omez] \label{topemb} A loop system of finite topological entropy can be continuously embedded in an invertible  compact topological Markov shift. 
\end{theorem}

\begin{lemma} \label{lemdisj} Let $A$ and $B$ be disjoint open subsets of $\Sigma_a$. Suppose $A_1$  and $B_1$ are open subsets of $\Sigma$ such that $A_1\cap \Sigma_a=A$ and $B_1\cap\Sigma_a=B$. Then $A_1$ and $B_1$ are disjoint. 
\end{lemma}
\begin{proof}
Suppose that $A_1$ and $B_1$ are not disjoint. In this case we can find a non-empty open set $U\subset A_1\cap B_1$. Define $V=U\cap \Sigma_a$ and observe that $V\subset A\cap B$. Since $\Sigma_a$ is dense in $\Sigma$ we have that $V=U\cap\Sigma_a$ is non-empty, which contradicts that $A$ and $B$ are disjoint. 
\end{proof}

Given an open subset $A$ of $\Sigma_a$ we define $\widehat{A}$ as the largest open subset in $\Sigma$ such that $\widehat{A}\cap\Sigma_a=A$. Similarly, for a closed subset $B$ of $\Sigma_a$ we define $\widecheck{B}$ as the smallest closed subset of $\Sigma$ such that $\widecheck{B}\cap \Sigma_a=B$. The existence of both $\widehat{A}$ and $\widecheck{B}$ follows from Zorn's Lemma. 

\begin{remark} \label{remequa} If $B \subset \Sigma_a$ is closed set, then there exists a closed set $B_1 \subset \Sigma$ such that $B_1\cap \Sigma_a=B$. Since $B\subset B_1$, we conclude that
$\overline{B}\subset B_1$, where $\overline{B}$ is the closure of $B$ in $\Sigma$. This implies that $B\subset \overline{B}\cap \Sigma_a\subset B_1\cap \Sigma_a=B$, and therefore $\overline{B}\cap \Sigma_a=B$. Moreover   $\widecheck{B}=\overline{B}$.
\end{remark}

\begin{lemma}\label{lemincl} Let $A$ be an open and closed subset of $\Sigma_a$. Then $\widehat{A}\subset \widecheck{A}$, in particular $\overline{\widehat{A}}=\overline{A}$. 
\end{lemma}
\begin{proof} Let $B:=\Sigma_a\setminus A$. Observe that  $\Sigma\setminus \widecheck{A}$ is open and that $(\Sigma\setminus \widecheck{A})\cap \Sigma_a=B$. By the definition of $\widehat{B}$ it follows that $(\Sigma\setminus \widecheck{A})\subset \widehat{B}$, or equivalently that $\Sigma\setminus \widehat{B}\subset \widecheck{A}$. Observe that $A$ and $B$ are disjoint open subsets of $\Sigma_
a$, therefore we can use Lemma \ref{lemdisj} and obtain that $\widehat{A}\subset (\Sigma\setminus \widehat{B})$. All this together implies that $\widehat{A}\subset \widecheck{A}$. 
\end{proof}

\begin{lemma} \label{lempartition} Suppose that $\mathcal{R}=\{R_1,...,R_N\}$ is a partition of $\Sigma_a$ such that
the sets $R_i$, with $i \in \{1, \dots, N\}$,  are open and closed in the topology of $\Sigma_a$. Then there exists a finite partition $\widehat{\mathcal{R}}$ of $\Sigma$ which induces the partition $\mathcal{R}$ on $\Sigma_a$ and $\mu(\partial \widehat{\mathcal{R}})=0$ for every probability measure on $\Sigma$ such that $\mu(\Sigma_a)=1$. 
\end{lemma}

\begin{proof} Let  $\widehat{\mathcal{R}}=\{\widehat{R_1},...,\widehat{R_N}, X\}$, where $X=\Sigma\setminus\bigcup_{i=1}^N\widehat{R_i}$. The partition is well defined since by Lemma \ref{lemdisj} the sets $\{\widehat{R_1},...,\widehat{R_N} \}$ are disjoint. Observe that  the set $X$ is closed and has empty interior (since $\Sigma_a$ is dense in $\Sigma$). Therefore $\mu(\partial X)=\mu(X\setminus \text{ int }X)=\mu(X)\le  \mu(\Sigma\setminus \Sigma_a)=0$. 
It follows from Lemma \ref{lemincl}  that $\mu(\partial \widehat{R_i})=\mu(\overline{\widehat{R_i}}\setminus \widehat{R_i})=\mu(\overline{R_i}\setminus \widehat{R_i})$. As observed in Remark \ref{remequa} we have that $\overline{R_i}\cap \Sigma_a=R_i$. Therefore
\begin{equation*}
\mu(\partial\widehat{R_i})=\mu(\overline{R_i}\setminus \widehat{R_i})=\mu((\overline{R_i}\cap \Sigma_a)\setminus \widehat{R_i})=\mu(R_i\setminus\widehat{R_i})=0.
\end{equation*}
\end{proof}

\begin{proposition} \label{lem} Let $(\mu_n)_n$ be a sequence of invariant probability measures converging in the weak$^*$ topology to a measure $\mu$. If $\mu(\Sigma_a)=1$ and  $\mu_n(\Sigma_a)=1$, for every      
 $n\in \N$, then 
\begin{equation*}
\limsup_{n\to \infty} h_{\mu_n}(\sigma)\le h_\mu(\sigma).
\end{equation*}
\end{proposition}
\begin{proof} 

It was shown in Lemma \ref{topconj} that there exists a topological conjugacy $F:\Sigma_a\to \overline{\Sigma}$ between $(\Sigma_a, \sigma)$ and $(\overline{\Sigma}, \sigma)$. From  Theorem  \ref{topemb}
there exists a topological embedding $G: \overline{\Sigma}\hookrightarrow \Sigma_0$, where $\Sigma_0$ is a sub-shift of finite type with alphabet $\{1,...,N\}$.  The partition $\mathcal{P}=\{[1],[2],...,[N]\}$ is a generating partition of $(\Sigma_0,\sigma)$  for every invariant probability measure (see \cite[Theorem 8.2]{wa}).  It follows that $\mathcal{Q}=G^{-1} \mathcal{P}$
 is a generating partition for every $\sigma$-invariant measure in $\overline{\Sigma}$.   The continuity of $G$ implies that  $G^{-1}\left( [i] ) \right)$ is open and closed, in particular  
\begin{equation*}
\partial G^{-1}\left( [i]  \right)=\overline{G^{-1}\left( [i] \right)  }\setminus \text{ int } G^{-1}\left( [i]  \right)  =\emptyset.
\end{equation*}
Therefore $\partial \mathcal{Q}=\emptyset$. Let $\mathcal{R}=F^{-1}\mathcal{Q}$. Then $\mathcal{R}$ is a generating partition for $(\Sigma_a,\sigma)$, $\partial \mathcal{R}=\emptyset$ and moreover the elements in $\mathcal{R}$ are open and closed. From Lemma \ref{lempartition} we construct a partition $\widehat{\mathcal{R}}$ of $\Sigma$  that induces $\mathcal{R}$ when restricted to $\Sigma_a$. Since by assumption the measures $\mu$ and $(\mu_n)_n$ give full measure to $\Sigma_a$,  Lemma \ref{lempartition} implies that 
$\mu(\partial \widehat{\mathcal{R}})=0$. Moreover $h_{\mu}(\sigma)=h_{\mu}(\sigma,  \widehat{\mathcal{R}})$ and $h_{\mu_n}(\sigma)=h_{\mu_n}(\sigma,  \widehat{\mathcal{R}})$ for every $n\in\N$. Indeed, for every $\sigma$-invariant probability measure $\nu$  on $\Sigma$ such that $\nu(\Sigma_a)=1$ we have that the systems $(\Sigma, \sigma, \nu)$ and $(\Sigma_a, \sigma |_{\Sigma_a}, \nu |_{\Sigma_a})$ are isomorphic. Therefore 
\begin{equation*}
h_{\nu}(\Sigma, \sigma)=h_{\nu}(\Sigma_a, \sigma) = \lim_{n \to \infty} \frac{1}{n} H_{\nu}(\mathcal{R}^n),
\end{equation*}
since $\mathcal{R}$ is a generating partition for $(\Sigma_a, \sigma)$. Note that for every $n \in \N$ and $\widehat{\mathcal{R}}_{i_1 , \dots , \i_n}  \in \widehat{\mathcal{R}}^n$ we have 
\begin{equation*}
\nu(\widehat{\mathcal{R}}_{i_1 , \dots , \i_n} )= \nu(\widehat{\mathcal{R}}_{i_1 , \dots , \i_n}  \cap \Sigma_a) = \nu|_{\Sigma_a} (\mathcal{R}_{i_1 , \dots , \i_n} ).
\end{equation*}
Thus $h_{\nu}(\sigma)= h_{\nu}(\sigma, \widehat{\mathcal{R}})$. Therefore, Proposition \ref{prop:usc} implies that
\begin{equation*}
\limsup_{n\to\infty} h_{\mu_n}(\sigma)=\limsup_{n\to \infty} h_{\mu_n}(\sigma,\widehat{\mathcal{R}})\le h_\mu(\sigma,\widehat{\mathcal{R}})=h_\mu(\sigma).
\end{equation*}
\end{proof}

\begin{proof}[Proof of  Theorem \ref{thm:usc}]
The above argument proves that if $(\mu_n)_n$ is a sequence of ergodic measures that converges in the weak$^*$ topology to the ergodic measure $\mu$ then 
$\lim_{n \to \infty} h_{\mu_n}(\sigma) \leq h_{\mu}(\sigma)$. Indeed, there exists $a \in \mathcal{A}$ such that $\mu([a]) >0$. Since $\lim_{n \to \infty} \mu_n([a])= \mu([a])>0$, there exists  $N \in \N$ such that for every $n >N$ 
we have that $\mu_n([a])>0$. We then use Proposition \ref{lem} to get the result. If $(\Sigma,\sigma)$ is not transitive we argue similarly. Since $\mu$ is ergodic and $\mu([a])>0$ we know that $\mu$ must be supported in the transitive component containing $\Sigma_a$. Moreover, $\mu$ is supported in a transitive countable Markov shift, say $(\Sigma_0,\sigma)$, which is the countable Markov shift associated to the connected component of the directed graph of $(\Sigma,\sigma)$ containing the vertex $a$. For large enough $n$ we will have $\mu_n([a])>0$, and therefore $\mu_n$ is supported in $\Sigma_0$. We then use Proposition \ref{lem} to conclude the result. 
\end{proof}

\subsection{One-sided countable Markov shifts}

Let $(\Sigma^+, \sigma)$ be a  one-sided Markov shift defined over a countable alphabet $\mathcal{A}$. This means that there exists a matrix $S=(s_{ij})_{\mathcal{A} \times \mathcal{A} }$ of zeros and ones such that
\begin{equation*}
\Sigma^+=\left\{ x \in \mathcal{A} ^{\mathbb{N}} : s_{x_{i} x_{i+1}}=1 \ \text{for every $i \in \mathbb{N}$}\right\}.
\end{equation*}
The \emph{shift map} $\sigma:\Sigma^+ \to \Sigma^+$ is defined by $(\sigma(x))_i=x_{i+1}$. The entropy of $\sigma$ is defined by $h_{top}=h_{top}(\sigma):=  \sup \left\{h_{\mu}(\sigma) : \mu \in \mathcal{E}_{\sigma}	\right\}$. Note that, as in the two-sided setting, it is possible to give a definition of entropy computing the exponential growth of periodic orbits, but for the purposes of this article the above definition suffices.

\begin{proof}[Proof of Corollary~\ref{cor:uscN}]

Denote by $(\Sigma,\sigma)$ the natural extension of $(\Sigma^{+},\sigma)$ and by $\mathcal{E}$ and $\mathcal{E}^{+}$
the corresponding sets of ergodic measures. There exists a bijection $\pi: \mathcal{E}^{+} \to \mathcal{E}$ such that for every $\mu \in \mathcal{E}^{+}$ we have that $h_{\mu}= h_{\pi(\mu)}$ (see \cite[Fact 4.3.2]{d} and \cite[Section 2.3]{sa2}).

\begin{lemma} \label{lem:conv}
Let $(\mu_n)_n, \mu \in \mathcal{E}^{+}$  such that $(\mu_n)_n$ converges weak$^*$  to $\mu$.
Then $(\pi(\mu_n))_n$ converges weak$^*$  to $\pi(\mu)$.\end{lemma}

\begin{proof}
Note that in the notion of  weak$^*$   convergence we can replace the set of test functions  by the space  of bounded  Lipschitz functions (see Remark \ref{rem:lip}). That is, if for every  bounded  Lipschitz function $\phi: \Sigma \to \R$ we have 
\begin{equation*}
\lim_{n \to \infty} \int \phi~d \mu_n = \int \phi ~d \mu,
\end{equation*}
then the sequence $(\mu_n)_n$ converges in the weak$^*$  topology to $\mu$.  A result by Daon \cite[Theorem 3.1]{da} implies that for every Lipschitz  (the result also holds for weakly H\"older and summable variations) function $\phi:\Sigma \to \R$ there exists a cohomologous Lipschitz function
$\psi:\Sigma \to \R$ that depends only on future coordinates. The function $\psi$ can be canonically identified with a Lipschitz function $\rho:\Sigma^+ \to \R$. Thus,
\begin{equation*}
\int_{\Sigma} \phi ~d \pi(\mu_n)= \int_{\Sigma} \psi ~d \pi(\mu_n)=\int_{\Sigma^+} \rho ~d \mu_n.
\end{equation*}
Therefore,
\begin{equation*}
\lim_{n\to \infty }\int_{\Sigma} \phi ~d \pi(\mu_n) = \lim_{n \to \infty}\int_{\Sigma^+} \rho ~d \mu_n= \int_{\Sigma^+} \rho ~d \mu =
\int_{\Sigma} \phi ~d \pi(\mu_n).
\end{equation*}
The result now follows.
\end{proof}

Let $\mu_n, \mu \in \mathcal{E}_{\sigma}$ be  such that $(\mu_n)_n$ converges weak$^*$ to $\mu$. Lemma \ref{lem:conv} implies that $(\pi(\mu_n))_n$ converges weak$^*$ to $\pi(\mu)$. Moreover, for every $n \in \N$ we have that $h_{\pi(\mu_n)}=h_{\mu_n}$  and $h_{\pi(\mu)}=h_{\mu}$. Since, by Theorem \ref{thm:usc}, the entropy map is upper semi-continuous in  $(\Sigma,\sigma)$ we obtain the result.
\end{proof}

\begin{remark} \label{rem:infentropy}
The finite entropy assumption in Theorem  \ref{thm:usc} and in Corollary~\ref{cor:uscN} is essential as the following example shows.
Let $(\Sigma^+, \sigma)$ be the full shift on a countable alphabet, note that $h_{top}(\sigma)=\infty$. Denote by $\mathcal{P}$ the partition formed by the length one cylinders. This is a generating partition. Let $h \in \R^+$ be a positive real number and  $(a_n)_n$ be the sequence defined by $a_n= \frac{h}{\log n}$ for every $n >1$. Consider the following stochastic vector
\begin{equation*}
\vec{p}_n:= \left(1-a_n, \frac{a_n}{n} , 	 \frac{a_n}{n}, \dots,  \frac{a_n}{n}, 0, 0, \dots	\right),
\end{equation*} 
where the term $a_n/n$ appears $n$ times.  Let $\mu_n$ be the Bernoulli measure defined by $\vec{p}_n$. Note that the sequence $(\mu_n)_n$ converges in the weak$^*$  topology to a Dirac measure  $\delta_1$ supported on the fixed point at the cylinder $C_1$. Note that 
\begin{eqnarray*}
H_{\mu_n}(\mathcal{P})= -(1-a_n)\log(1-a_n) -a_n \log a_n -a_n \log \frac{1}{n}. 
  \end{eqnarray*}
Therefore,
\begin{eqnarray*}
\lim_{n \to \infty} H_{\mu_n}(\mathcal{P}) = h > 0 = H_{\delta_1}(\mathcal{P}).
\end{eqnarray*}
The above example shows that the map $\nu \to H_{\nu}(\mathcal{P})$ need not to be continuous for countable generating partitions. Moreover, since the measures $\mu_n$ are Bernoulli we have that $h_{\mu_n}(\sigma)=H_{\mu_n}(\mathcal{P})$, and therefore the above argument shows that the entropy map is not upper semi-continuous:
\begin{eqnarray*}
\lim_{n \to \infty} h_{\mu_n}(\sigma) = h > 0 = h_{\delta_1}(\sigma). 
\end{eqnarray*}
In other words, the entropy map could fail to be upper semi-continuous for general dynamical systems defined on non-compact spaces, even if there exists a uniform (countable) generating partition with no boundary. The construction of the above example is based on examples constructed by Walters (\cite[p.184]{wa}) and by Jenkinson, Mauldin and Urba\'nski \cite[p.774]{jmu}. This example also shows that the entropy map is not upper semi-continuous even if we consider a sequence of measures for which their entropy is uniformly bounded.
\end{remark}

\section{Measures of maximal entropy} \label{sec:mme}
A continuous map $T:(X,d ) \to (X,d)$ defined on a compact metric space for which the entropy map is upper semi-continuous has a measure of maximal entropy. Indeed, from the variational principle there exists a sequence of ergodic invariant probability measures $(\mu_n)_n$ such that $\lim_{n \to \infty} h_{\mu_n}(T)= h_{top}(T)$. Since the space of invariant measures $\M_T$ is compact, there exists an invariant  measure $\mu$ which is an accumulation point for $(\mu_n)_n$. It follows from the fact that the entropy map is upper semi-continuous  that
\begin{equation*}
h_{top}(T)= \lim_{n \to \infty} h_{\mu_n}(T) \leq h_{\mu}(T).
\end{equation*}
Therefore, $\mu$ is a measure of maximal entropy. For countable Markov shifts the variational principle holds (see \cite{gu1,gu2}) and the entropy map is upper semi-continuous (see Theorem~\ref{thm:usc} and Corollary~\ref{cor:uscN}), however the space $\M_{\sigma}$ is no longer compact.  Despite this, under a convergence assumption we can prove the existence of measures of maximal entropy. Indeed, Corollary~\ref{cor:uscN} provides a new proof of the following result by Gurevich and Savchenko \cite[Theorem 6.3]{gs}.

\begin{proposition} \label{prop:mme}
Let $(\Sigma, \sigma)$ be a finite entropy countable Markov shift. Let $(\mu_n)_n$ be a sequence of ergodic measures
such that $\lim_{n \to \infty} h_{\mu_n}(\sigma) =h_{top}(\sigma)$. If $(\mu_n)_n$ converges in the  weak$^*$  topology to an ergodic measure $\mu$ then
$h_{\mu}(\sigma)=h_{top}(\sigma)$.
\end{proposition}

It might happen that there is a sequence $(\mu_n)_n$  with $\lim_{n \to \infty} h_{\mu_n}(\sigma)=h_{top}(\sigma)$, but $(\mu_n)_n$ does not converge in the weak$^*$  topology. Examples of countable Markov shifts with this property have been known for a long time. In \cite{ru2} a simple construction of a Markov shift of any given entropy  with no measure of maximal entropy is provided. In \cite{gu3,gn} examples are constructed of finite entropy countable Markov shifts having a measure of maximal entropy $\mu_{max}$ and sequences of ergodic measures $(\mu_n)_n , (\nu_n)_n$ with $\lim_{n \to \infty} h_{\mu_n}(\sigma) = \lim_{n \to \infty} h_{\nu_n}(\sigma) =h_{top}(\sigma)$ such that $(\mu_n)_n$ converges  in the weak$^*$  topology to $\mu_{max}$ and $(\nu_n)_n$ does not have any accumulation point. 

The continuity properties of the entropy map also have consequences in the study of thermodynamic formalism. Let $(\Sigma^+, \sigma)$ be a transitive one-sided countable Markov shift of finite entropy and $\phi:\Sigma^+ \to \R$ a continuous bounded function of summable variations. That is $\sum_{n=1}^{\infty} \text{var}_n(\phi) < \infty$, where 
\begin{equation*}
\text{var}_n(\phi):= \sup\left\{\phi(x) - \phi(y) : x_i=y_i , i \in \{1, \dots, n\}		\right\}.
\end{equation*}
Sarig (see \cite{sa2} for a survey on the topic) defined a notion of pressure in this context, the so called Gurevich pressure, that we denote by $P(\phi)$. He proved the following variational principle
\begin{equation*}
P(\phi)= \sup \left\{ h_\mu(\sigma)+ \int \phi~d \mu : \mu \in \M_{\sigma} \right\}=  \sup \left\{ h_\mu(\sigma) + \int \phi~d \mu : \mu \in \mathcal{E}_{\sigma} \right\}
\end{equation*}
A measure $\mu \in \M_{\sigma}$ such that $P(\phi)= h_\mu(\sigma) + \int \phi~d \mu$ is called \emph{equilibrium measure}. It directly follows from Corollary~\ref{cor:uscN} that

\begin{proposition} \label{prop:equ}
Let $(\Sigma^+, \sigma)$ be a finite entropy countable Markov shift and $\phi:\Sigma \to \R$ a bounded function of summable variations.
 Let $(\mu_n)_n$ be a sequence of ergodic measures
such that $\lim_{n \to \infty} \left( h_{\mu_n}(\sigma) + \int \phi~d \mu_n   \right) =P(\phi)$. If $(\mu_n)_n$ converges in the weak$^*$   topology to an ergodic measure $\mu$ then
$P(\phi)=  h_{\mu}(\sigma) + \int \phi~d \mu$.
\end{proposition}
Note that the thermodynamic formalism of a two-sided countable Markov shift can be reduced to the one-sided case (see \cite[Section 2.3]{sa2}).

\section{Suspension  flows} \label{sec:flo}

Let $(\Sigma, \sigma)$ be a finite entropy  countable Markov shift and let $\tau: \Sigma \to \R^+$ be a locally H\"older potential 
bounded away from zero. Consider the space $Y= \left\{ (x,t)\in \Sigma  \times \R \colon 0 \le t \le\tau(x) \right\}$,
with the points $(x,\tau(x))$ and $(\sigma(x),0)$ identified for each $x\in \Sigma $. The \emph{suspension flow} over $\Sigma$
with \emph{roof function} $\tau$ is the semi-flow $\Phi= (\phi_t)_{t \in \R}$ on $Y$ defined by
$ \phi_t(x,s)= (x,s+t)$ whenever $s+t\in[0,\tau(x)]$. Denote  by $\mathcal{M}_{\Phi}$ the space of flow invariant probability measures. Let 
\begin{equation}
\mathcal{M}_\sigma(\tau):= \left\{ \mu \in \mathcal{M}_{\sigma}: \int \tau ~d \mu < \infty \right\}.
\end{equation}
It follows directly from results by Ambrose and Kakutani \cite{ak} that the map $R \colon \mathcal{M}_\sigma \to \mathcal{M}_\Phi$, defined by
\begin{equation*} \label{eq:R map}
R(\mu)=\frac{(\mu \times  \text{Leb})|_{Y} }{(\mu \times  \text{Leb})(Y)},
\end{equation*}
where  \text{Leb} is the one-dimensional Lebesgue measure, is a bijection. Denote by $\mathcal{E}_{\Phi}$ the set of ergodic flow invariant measures.   Let  $F \colon Y\to\R$ be a continuous function. Define
$\Delta_F\colon\Sigma \to\R$~by
\[
\Delta_F(x):=\int_{0}^{\tau(x)} F(x,t) ~dt.
\]
Kac's Lemma states that if  $\nu \in \M_{\Phi}$ is an invariant measure  that can be written as
$$\nu=\frac{\mu \times  \text{Leb}} {(\mu \times  \text{Leb})(Y)},$$
where $\mu \in\M_{\sigma}$, then
\begin{equation*} \label{eq:rela}
\int_{Y}F ~d \nu= \frac{\int_\Sigma \Delta_F ~d\mu}{\int_\Sigma\tau~d \mu}.
\end{equation*}
The following Lemma describes the relation between weak$^*$  convergence in $\M_{\Phi}$ with that in $\M_{\sigma}$.

\begin{lemma} \label{lem:weak} 
Let $(\nu_n), \nu \in \M_{\Phi}$ be flow invariant probability measures such that
\begin{equation*}
\nu_n=\frac{\mu_n \times Leb}{\int \tau~d \mu_n} \quad \text{ and } \quad \nu= \frac{\mu_n \times Leb}{\int \tau~d \mu_n} 		
\end{equation*}
where $(\mu_n)_n , \mu \in \M_{\sigma}$ are shift invariant probability measures.  The sequence $(\nu_n)_n$ converges  in the weak$^*$  topology  to $\nu$ then
\begin{equation*}
(\mu_n)_n \text{ converges in the weak$^*$  topology to } \mu \quad \text{  and  } \quad \lim_{n \to \infty} \int \tau~d \mu_n = \int \tau~d\mu. 		
\end{equation*}
\end{lemma}

\begin{proof} Assume first that $(\nu_n)_n$ converges in the weak$^*$  topology to $\nu$. Let $f :\Sigma \to \R$ be a bounded continuous function. Following Barreira, Radu and Wolf \cite{brw} there exist a continuous function $F:Y \to \R$ such that
\begin{equation} \label{eq:f}
f(x)= \Delta_{F}(x):= \int_0^{\tau(x)} F(x,t)~dt.
\end{equation} 
Indeed, define $F(x,t): Y \to \R$ by
\begin{equation*}
F(x,t):= \frac{f(x)}{\tau(x)} \psi' \left( \frac{t}{\tau(x)}		\right),
\end{equation*}
where $\psi:[0,1] \to [0,1]$ is a $C^1$ function such that $\psi(0)=0$, $\psi(1)=1$ and $\psi'(0)=\psi'(1)=0$. Note that since $\tau$ is bounded away from zero, $F(x,t)$ is continuous and bounded. Therefore
\begin{equation*}
\lim_{n \to \infty} \int F~d \nu_n = \int F~d \nu.
\end{equation*}
By Kac's Lemma we have that
\begin{equation} \label{eq:k}
\lim_{n \to \infty} \frac{\int \Delta_F~d \mu_n}{\int \tau~d \mu_n} =\frac{\int \Delta_F~d \mu}{\int \tau~d \mu}.
\end{equation}
In particular if $f=1$ is the constant function equal to one, we obtain
\begin{equation} \label{eq:ratio}
\lim_{n \to \infty} \frac{\int \tau~d \mu_n}{\int \tau~d \mu}=1.
\end{equation}
Let  $f :\Sigma \to \R$ be a bounded continuous function, then it follows from equations \eqref{eq:f},  \eqref{eq:k}  and \eqref{eq:ratio} that
\begin{equation*}
\lim_{n \to \infty} \int f~d \mu_n =\int f~d \mu.
\end{equation*}
Therefore $(\mu_n)$ converges weak$^*$  to $\mu$.
\end{proof}

\begin{proposition} \label{thm:susp}
 If $(\nu_n)_n$ is a sequence of ergodic measures in $\mathcal{E}_{\Phi}$ converging to an ergodic measure  $\nu$, then 
\begin{equation*}
 \limsup_{n \to \infty}  h_{\nu_n}(\Phi) \leq h_{\nu}(\Phi).
\end{equation*}
\end{proposition}

\begin{proof}
Let $(\nu_n)_n$ be a sequence in $\mathcal{E}_{\Phi}$ that converges in the weak$^*$  topology to $\nu \in \mathcal{E}_{\Phi}$. Note that there exists a sequence of ergodic measures $(\mu_n)_n \in  \mathcal{M}_{\sigma}$ and $\mu \in \mathcal{M}_{\sigma}$ ergodic such that
\begin{equation*}
\nu_n= \frac{\mu_n \times \text{Leb}}{(\mu_n \times \text{Leb})(Y)} \ \text{ and }  \  \nu= \frac{\mu \times \text{Leb}}{(\mu_n \times \text{Leb})(Y)}.\end{equation*}
By Abramov's formula \cite{a}, for any $\nu' \in \mathcal{M}_{\Phi}$, with $\nu'= \frac{\mu' \times \text{Leb}}{(\mu' \times \text{Leb})(Y)} $ we have that
\begin{equation*}
h_{\nu'}(\Phi)=\frac{h_{\mu'}(\sigma)}{\int \tau~d \mu'}.
\end{equation*}
Recall that by Lemma \ref{lem:weak} we have that  $\nu_n \to \nu$ in the weak$^*$   topology  implies that $\mu_n \to \mu$  in the weak$^*$  topology and 
$ \lim_{n \to \infty} \int \tau ~d \mu_n = \int \tau ~d \mu.$   Since $h_{top}(\sigma) < \infty$ by Theorem~\ref{thm:usc} or Corollary~\ref{cor:uscN} we have that 
\begin{equation*}
\limsup_{n \to \infty} h_{\mu_n}(\sigma) \leq h_{\mu}(\sigma).
\end{equation*}
Therefore, 
\begin{eqnarray*}
\limsup_{n \to \infty} h_{\nu_n}(\Phi) = \limsup_{n \to \infty} \frac{h_{\mu_n}(\sigma)}{\int \tau ~d \mu_n } =\frac{ \limsup_{n \to \infty}  h_{\mu_n}(\sigma)}{\int \tau ~d \mu} \leq \frac{h_{\mu}(\sigma)}{\int \tau ~d \mu}= h_{\nu}(\Phi).
\end{eqnarray*}

\end{proof}


\end{document}